\documentclass{amsart}
\usepackage{graphicx}
\usepackage{amsmath}
\usepackage{amssymb}%
\usepackage{amsfonts}%
\usepackage{graphicx}
\usepackage[all]{xy}
\usepackage[mathscr]{euscript}
\usepackage{hyperref}
\newtheorem{theorem}{Theorem}[section]
\newtheorem{lemma}[theorem]{Lemma}

\theoremstyle{definition}
\newtheorem{definition}[theorem]{Definition}
\newtheorem{example}[theorem]{Example}
\newtheorem{proposition}[theorem]{Proposition}
\newtheorem{corollary}[theorem]{Corollary}

\usepackage[all]{xy}
\usepackage{graphicx}
\theoremstyle{remark}
\newtheorem{remark}[theorem]{Remark}

\numberwithin{equation}{section}
\usepackage{graphicx}

\begin{document}
\title{Length Minimising Bounded Curvature Paths in Homotopy Classes}
\author{Jos\'{e} Ayala}
\address{FIA, Universidad Arturo Prat, Iquique, Chile}
\email{jayalhoff@gmail.com}
\maketitle

\begin{abstract} Choose two points in the tangent bundle of the Euclidean plane $(x,X),(y,Y)\in T{ \mathbb R}^2$. In this work we characterise the immersed length minimising paths with a prescribed bound on the curvature starting at $x$, tangent to $X$; finishing at $y$, tangent to $Y$, in each connected component of the space of paths with a prescribed bound on the curvature from $(x,X)$ to $(y,Y)$.
\end{abstract}

\section{Introduction}
A planar bounded curvature path corresponds to a $C^1$ and piecewise $C^2$ path lying in ${\mathbb R}^2$ having its curvature bounded by a positive constant and connecting two elements of the tangent bundle $T{\mathbb R}^2$. Length minimising bounded curvature paths, widely known as Dubins paths, have proven to be extraordinarily useful in robotics since a bound on the curvature is a turning circle constraint for the trajectory of a robot along a path. Bounded curvature paths admitting self intersections arise naturally in applications, for example, in the design of a system of interconnected tunnels navigated by vehicles. These paths are important as projections of embedded 3-dimensional paths (see \cite{brazil 1}). Counterintuitively, bounded curvature paths admitting self intersections may be length minimisers for certain choices of initial and final points in $T{\mathbb R}^2$. This work corresponds to the culmination of our study on minimal length elements in spaces of bounded curvature paths in $\mathbb R^2$ and evokes some of the machinery developed in \cite{papera}, \cite{paperc} and \cite{paperd}. In \cite{papera} we developed a method for finding length minimising bounded curvature paths for any given initial and final points and directions as follows. We start with an arbitrary bounded curvature path and divide it up into a finite number of pieces of suitable length called  {\it fragments}. Then, for each fragment, we construct a piecewise constant curvature path (close to the fragment) of length at most the length of the fragment. Then, we proceed to {\it replace} the fragments by piecewise constant curvature paths\footnote{We avoid using homotopy since the replacement does not involves a continuous one-parameter family of paths.}. The process of replacing a bounded curvature path by a piecewise constant curvature path is called {\it replacement}. After all of the fragments have been replaced what is left is a concatenation of piecewise constant curvature paths (possibly arbitrarily close to the original path) called {\it normalisation}. Subsequently, we developed a series of results involving larger pieces of the normalisation, and again, we replace such pieces (or components) by piecewise constant curvature paths of less  {\it complexity} observing that a path that can be perturbed while decreasing its length cannot be candidate to be length minimiser; we refer to this as {\it reduction process}. After a finite number of steps we obtain the characterisation of the global length minimisers, a well known result obtained by Dubins in \cite{dubins 1}. A crucial observation is that the act of replacing a fragment by a piecewise constant curvature path does not involve a continuity argument, and therefore, without an explicit homotopy preserving the bound on the curvature between these paths nothing can be said in terms of homotopy classes. This rather technical issue is addresses in \cite{paperd} were we classified the homotopy classes of bounded curvature paths for any given initial and final points in $T {\mathbb R}^2$.  In \cite{paperc} we proved that for certain initial and final points there exist a homotopy class whose elements are embedded. These paths cannot be deformed to paths with self intersections without violating the curvature bound.

In this work we develop a machinery to justify following procedure. For a bounded curvature path (possibly with loops) lying in a prescribed homotopy class we continuously deform the path into a piecewise constant curvature path by applying the continuity argument developed in \cite{paperd} so we make sure both paths lie in the same connected component. Then, we homotope the obtained piecewise constant curvature path to a shorter one while reducing the number of arcs of circle and/or line segments (complexity of the path) taking special care into components with loops. After applying this process a finite number of times we achieve the desired characterisation for the minimum length elements in spaces of planar bounded curvature paths with prescribed {\it turning number}. Our main result, Theorem \ref{singudub}, characterises in particular the global length minimisers generalising Dubins result in \cite{dubins 1}. The methodology used to prove our main result can be seen in Figure \ref{figsingmovccc}. Observe that a global length minimiser may not be unique and, in such a case, these paths are elements in different homotopy classes (see Figure \ref{figccproxcond1}). Many of the results exposed in this note have been verified with {\sc dubins explorer} a mathematical software for bounded curvature paths developed by Jean Diaz and the author (see \cite{dubinsexplorer}). In particular, the software computes the length minimisers and the number of connected components of the space of bounded curvature paths from $(x,X)$ to $(y,Y)$. Here we also give an elementary proof for: {The length of a closed loop whose absolute curvature is bounded by 1 is at least $2\pi$}. We suggest the reader to read this work in conjunction with \cite{papera}, \cite{paperc} and \cite{paperd}.

\section{preliminaries}
Denote by $T{\mathbb R}^2$ the tangent bundle of ${\mathbb R}^2$. Recall the elements in $T{\mathbb R}^2$ correspond to pairs $(x,X)$ sometimes denoted just by {\sc x}. Here the first coordinate corresponds to a point in ${\mathbb R}^2$ and the second to a tangent vector to ${\mathbb R}^2$ at $x$.

\begin{definition} \label{adm_pat} Given $(x,X),(y,Y) \in T{\mathbb R}^2$,  a path $\gamma: [0,s]\rightarrow {\mathbb R}^2$ connecting these points is a {\it bounded curvature path} if:
\end{definition}
 \begin{itemize}
\item $\gamma$ is $C^1$ and piecewise $C^2$.
\item $\gamma$ is parametrised by arc length (i.e $||\gamma'(t)||=1$ for all $t\in [0,s]$).
\item $\gamma(0)=x$,  $\gamma'(0)=X$;  $\gamma(s)=y$,  $\gamma'(s)=Y.$
\item $||\gamma''(t)||\leq \kappa$, for all $t\in [0,s]$ when defined, $\kappa>0$ a constant.
\end{itemize}
Of course, $s$ is the arc length of $\gamma$.

The first condition means that a bounded curvature path has continuous first derivative and piecewise continuous second derivative. For the third condition makes sense, without loss of generality, we extend the domain of $\gamma$ to $(-\epsilon,s+\epsilon)$ for $\epsilon>0$. The third item is called endpoint condition. The fourth item means that  bounded curvature paths have absolute curvature bounded above by a positive constant which can be choose to be $\kappa=1$. In addition, note that the first item is a completeness condition since the length minimising paths satisfying simultaneously the last three items in Definition \ref{adm_pat} are is $C^1$ and piecewise $C^2$, compare \cite{papera} or \cite{dubins 1}. Generically, the interval $[0,s]$ is denoted by $I$. Recall that path a $\gamma: I \rightarrow {\mathbb R}^2$ has a self intersection if there exists $t_1,t_2 \in I$, with $t_1\neq t_2$ such that $\gamma(t_1)=\gamma(t_2)$.

\begin{definition} A $cs$ path is a bounded curvature path corresponding to a finite number of concatenations of line segments (denoted by {\sc s}) and arcs of a unit circle (denoted by {\sc c}) see Figures \ref{figccproxcond1},
\ref{figfunlem} and \ref{figrep1} for examples. Taking into account the path orientation {\sc r} denotes a clockwise traversed arc and {\sc l} denotes a counterclockwise traversed arc. The number of line segments plus the number of circular arcs is called the {\it complexity} of the path.
\end{definition}

\begin{example} {A {\sc csc} path is a $cs$ path corresponding to a concatenation of an arc of a unit radius circle, followed by a line segment, followed by an arc of a unit radius circle (see Figure \ref{figccproxcond1}). A {\sc ccc} path is a $cs$ path corresponding to a concatenation of three arcs of unit radius.}  For {\sc csc} and {\sc ccc} paths we obtain six possible configurations given by {\sc lsl}, {\sc rsr}, {\sc lsr}, {\sc rsl}, {\sc lrl} and {\sc rlr}. We call the {\sc lsl}, {\sc rsr} symmetric paths and the {\sc lsr}, {\sc rsl} skew paths.
\end{example}

 \begin{definition} \label{admsp} Given $\mbox{\sc x,y}\in T{\mathbb R}^2$. The space of bounded curvature paths satisfying the given endpoint condition is denoted by $\Gamma(\mbox{\sc x,y})$. \end{definition}

  When a path is continuously deformed under parameter $p$ we reparametrise each of the deformed paths by its arc length. Thus $\gamma: [0,s_p]\rightarrow {\mathbb R}^2$ describes a deformed path at parameter $p$, with $s_p$ corresponding to its arc length. This idea will be applied in the next definition.

\begin{definition}  \label{hom_adm} Given $\gamma,\eta \in \Gamma(\mbox{\sc x,y})$. A {\it bounded curvature homotopy}  between $\gamma: [0,s_0] \rightarrow  {\mathbb R^2}$ and $\eta: [0,s_1] \rightarrow  {\mathbb R^2}$ corresponds to a continuous one-parameter family of immersed paths $ {\mathcal H}_t: [0,1] \rightarrow \Gamma(\mbox{\sc x,y})$ such that:
\begin{itemize}
\item ${\mathcal H}_t(p): [0,s_p] \rightarrow  {\mathbb R}^2$ for $t\in [0,s_p]$ is an element of $\Gamma(\mbox{\sc x,y})$ for all $p\in [0,1]$.
\item $ {\mathcal H}_t(0)=\gamma(t)$ for $t\in [0,s_0]$ and ${\mathcal H}_t(1)=\eta(t)$ for $t\in [0,s_1]$.
\end{itemize}
\end{definition}

\begin{remark}(\it On homotopy classes) Given $\mbox{\sc x,y}\in T{\mathbb R^2}$ then:
\end{remark}
\begin{itemize}
\item Two bounded curvature paths are {\it bounded-homotopic} if there exists a bounded curvature homotopy from one path to another. The previously described relation defined by $\sim$ is an equivalence relation.
\item A {\it homotopy class} in $\Gamma(\mbox{\sc x,y})$ corresponds to an equivalence class in $\Gamma(\mbox{\sc x,y})/\sim$.
\item A {\it homotopy class} is a maximal path connected set in $\Gamma(\mbox{\sc x,y})$.
\end{itemize}

\begin{definition}Let $\mbox{\sc C}_ l(\mbox{\sc x})$ be the unit circle tangent to $x$ and to the left of $X$. Analogous interpretations apply for $\mbox{\sc C}_ r(\mbox{\sc x})$, $\mbox{\sc C}_ l(\mbox{\sc y})$ and $\mbox{\sc C}_ r(\mbox{\sc y})$ (see Figure \ref{figccproxcond1}). These circles are called {\it adjacent circles}. Denote their centres with lowercase letters, so the centre of $\mbox{\sc C}_ l(\mbox{\sc x})$ is denoted by $c_l(\mbox{\sc x})$. 
\end{definition}

Next we analyse the interaction between the position of $x$ and $y$, the initial and final vectors $X$ and $Y$ and the curvature bound.  We look at the arrangements for the adjacent circles in $\mathbb R^2$. The ideas here discussed were considered in oder to prove the existence of a compact planar region (denoted by $\Omega)$ that {\it traps} embedded bounded curvature paths (see  \cite{paperc} and Figures \ref{figccproxcond1} and \ref{figfunlem}). The following conditions give important information about the topology and geometry of $\Gamma(\mbox{\sc x,y})$.

\begin{equation} d(c_l(\mbox{\sc x}),c_l(\mbox{\sc y}))\geq 4 \quad \mbox{and}\quad d(c_r(\mbox{\sc x}),c_r(\mbox{\sc y}))\geq4 \label{con_a}\tag{i}\end{equation}
 \begin{equation} d(c_l(\mbox{\sc x}),c_l(\mbox{\sc y}))< 4 \quad \mbox{and}\quad d(c_r(\mbox{\sc x}),c_r(\mbox{\sc y}))\geq 4 \label{con_b}\tag{ii} \end{equation}
  \begin{equation} d(c_l(\mbox{\sc x}),c_l(\mbox{\sc y}))\geq4 \quad \mbox{and}\quad d(c_r(\mbox{\sc x}),c_r(\mbox{\sc y}))< 4  \label{con_b'}\tag{iii} \end{equation}
   \begin{equation} d(c_l(\mbox{\sc x}),c_l(\mbox{\sc y}))< 4 \quad \mbox{and}\quad d(c_r(\mbox{\sc x}),c_r(\mbox{\sc y}))< 4 \label{con_c}\tag{iv} \end{equation}

The idea is to correlate the conditions (i), (ii), (iii) and (iv) with the number of connected components in $\Gamma({\mbox{\sc x,y}})$. 
{ \begin{figure} [[htbp]
 \begin{center}
\includegraphics[width=.8\textwidth,angle=0]{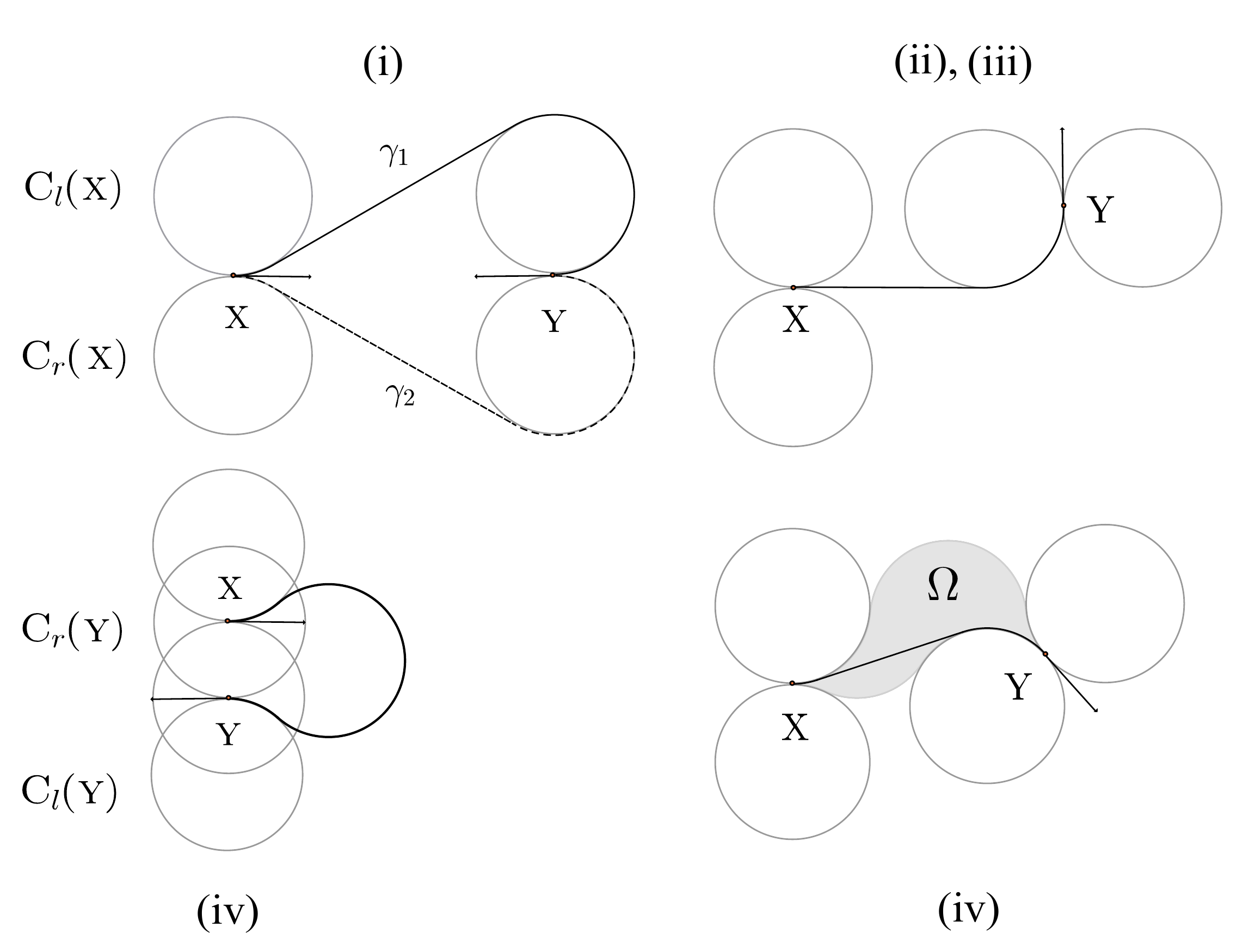}
\end{center}
\caption{Are the paths $\gamma_1$ and $\gamma_2$ in the same connected component? Here we show examples of configurations for the adjacent circles according to (i), (ii), (iii) and (iv). The illustrated paths are length minimisers in their respective connected component.}
 \label{figccproxcond1}
\end{figure}}

\begin{remark} {\it (On proximity conditions).}\hfill 
\end{remark}
\begin{itemize}
\item If the endpoint condition satisfies (i) we say that $\Gamma({\mbox{\sc x,y}})$ satisfies proximity condition {\sc A}.
\item  If the endpoint condition satisfies (ii) or (iii) we say that $\Gamma({\mbox{\sc x,y}})$ satisfies proximity condition {\sc B}.
\item If the endpoint condition satisfies (iv) and the elements in $\Gamma({\mbox{\sc x,y}})$ are bounded-homotopic to paths of arbitrary length we say that $\Gamma({\mbox{\sc x,y}})$ satisfies proximity condition {\sc C}.

\noindent {\it A crucial result in \cite{paperd} states that certain endpoint conditions give rise to homotopy classes of embedded bounded curvature paths; such paths are trapped in a planar compact region $\Omega$ as proved in \cite{paperc} (see Figures \ref{figccproxcond1} or \ref{figfunlem})}.

\item If the endpoint condition satisfies (iv) we say that $\Gamma({\mbox{\sc x,y}})$ satisfies proximity condition {\sc D} if $\Gamma({\mbox{\sc x,y}})$ has a homotopy class of embedded paths. 
\end{itemize}

Note that conditions {\sc A}, {\sc B} and {\sc C} lead to spaces $\Gamma({\mbox{\sc x,y}})$ whose elements are bounded-homotopic to paths arbitrarily long. An important result in \cite{paperc} says that paths inside $\Omega$ have bounded length. Here we will focus on conditions {\sc A, B, C} and {\sc D} to analyse the length minimisers in homotopy classes in $\Gamma({\mbox{\sc x,y}})$.

\section{Normalisation of paths and the homotopy argument}
 
 Next we describe the overall strategy we follow to prove Theorem \ref{singudub}, our main result. The concepts in this section can be found in detail in \cite{papera} and \cite{paperd}. We now introduce the {\it normalisation} of a bounded curvature path. The idea is to divide up any given bounded curvature path into {\it sufficiently small} pieces called {\it fragments}. Let ${\mathcal L}(\gamma,a,b)$ be the length of $\gamma:I \rightarrow {\mathbb R}^2$ restricted to $[a,b]\subset I$. We write ${\mathcal L}(\gamma,0,s)={\mathcal L}(\gamma)$. A {\it fragmentation} of $\gamma$ corresponds to a finite sequence $0=t_0<t_1\ldots <t_m=s$ such that, ${\mathcal L}(\gamma,t_{i-1},t_i)<  1$ with $\sum_{i=1}^m {\mathcal L}(\gamma,t_{i-1},t_i) =s$. A {\it fragment}, is the restriction of $\gamma$ to the interval determined by two consecutive elements in the fragmentation. 

  The following result shows the existence of a {\sc csc} path {close} to any given fragment. The critical argument here is that {\it fragments} are bounded homotopic to a  {\it replacement path}. Another crucial observation is that {the length of the fragment is at most the length of the replacement path}. The  {replacement path} is in fact the unique length minimiser in its homotopy class (see Figure \ref{figfunlem}).

 \begin{proposition} \label{construct} {\it (Proposition 2.13 in \cite{papera})} For a fragment in $\Gamma(\mbox{\sc x,y})$ there exists a {\sc csc} path in $\Gamma(\mbox{\sc x,y})$ having their circular components of length less than $\pi$.
 \end{proposition}

 In this work {\it sufficiently small} means that: {\it Given $\epsilon>0$ corresponding the length of a bounded curvature path $\gamma$, there exists $\delta>0$ such that the curvature at $\gamma(t)$ for $t\in (t-\delta,t+\delta)$ remains bounded by 1}. For sufficiently small fragments the bounded curvature homotopy is a projection of the fragment onto the replacement path (see Figure \ref{figfunlem}). 
 
 \begin{proposition}{\it \label{homotopyfragment} (Proposition 3.6 in \cite{paperd})} A sufficiently small fragment is bounded-homotopic to its replacement path.
\end{proposition}

 \begin{lemma}\label{fundlemma}  (Lemma 2.12. in \cite{papera}) The length of a replacement path is at most the length of the associated fragment with equality if and only if these paths are identical. (see Figure \ref{figfunlem}).
\end{lemma}
 
By applying Proposition \ref{homotopyfragment}  and Lemma \ref{fundlemma} we proved in \cite{paperd} the following result.

 \begin{theorem}\label{bhcs}(Theorem 3.7 in \cite{paperd}). A bounded curvature path is bounded-homotopic to a cs path.
\end{theorem}
{ \begin{figure} [[htbp]
 \begin{center}
\includegraphics[width=.9\textwidth,angle=0]{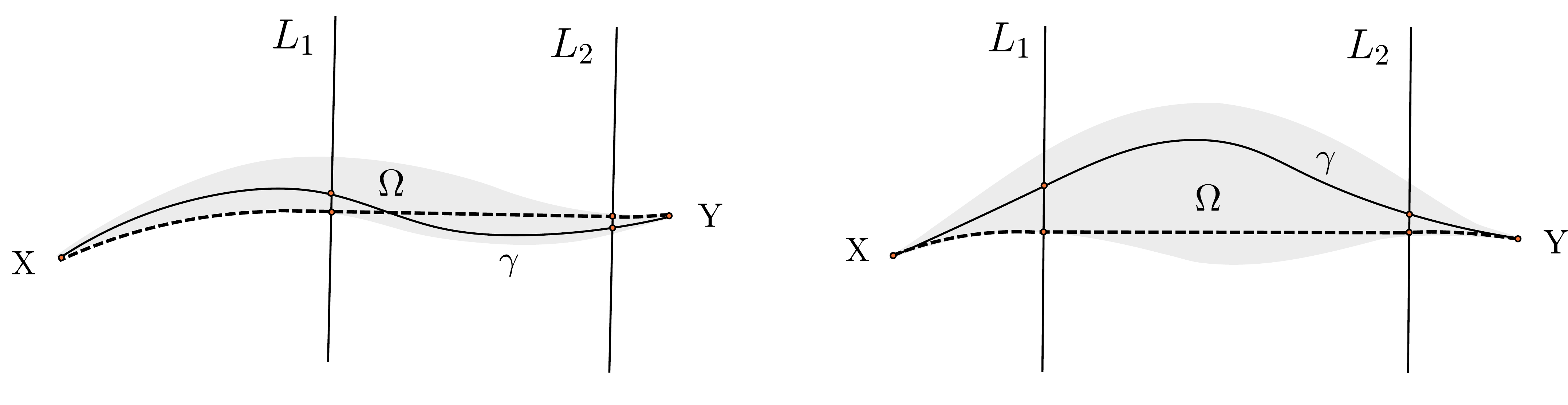}
\end{center}
\caption{Under condition {\sc D} embedded bounded curvature paths in $\Omega$ are not bounded-homotopic to paths having a point in the complement of $\Omega$. The dashed path corresponds to the {\sc csc} path bounded-homotopic to the fragment $\gamma$. For Lemma \ref{rad}: Note that $\gamma$ in between {\sc x} and $L_1$
 is longer than the dashed arc of a unit circle in between {\sc x} and $L_1$.}
 \label{figfunlem}
\end{figure}}

Here we put together Theorem \ref{bhcs} and Lemma \ref{fundlemma} to prove the following result.
 
  \begin{theorem} (Normalisation)\label{homotarg} A bounded curvature path $\gamma$ is bounded-homotopic to a cs path of length at most the length of $\gamma$.
\end{theorem}

\begin{proof} Consider a bounded curvature path with a fragmentation with sufficiently small fragments. By applying Theorem \ref{bhcs} we obtain a homotopy between $\gamma$ and a $cs$ path. The result follows by applying Lemma \ref{fundlemma} in between consecutive elements in the fragmentation. \end{proof}
 
Given a bounded curvature path we analyse its normalisation by pieces. These pieces (or {\it components}) are obtained by concatenating {\sc csc} paths. We denote by a component of type ${\mathscr C}_1$ a path of type {\sc cscsc} shown in Figure \ref{figrep1} left and centre. A component of type ${\mathscr C}_2$ is a path of type {\sc csccsc} shown in Figure \ref{figrep1} right. A component of type ${\mathscr C}_3$ contains a loop and will be defined later in this work. We have the following result.

\begin{theorem}(Theorem 3.4 in \cite{papera}) \label{scsccsnomin}Components of type $\mathscr{C}_1$ and $\mathscr{C}_2$ are not paths of minimal length.
\end{theorem}

A component is called {\it admissible} if it satisfies proximity condition {\sc A}, {\sc B} or {\sc D}. A {\it non-admissible} component satisfy proximity condition {\sc C}. A component is said to be {\it degenerate} if one of their sub arcs or line segments have length zero. The {\it reduction process} consist on continuously deform the components in a $cs$ paths to a $cs$ path with less complexity without increasing the length at any stage of the deformation see Figure \ref{figrep1}.  Our main result, Theorem \ref{singudub}, generalises the following well known result that characterises the length minimisers in $\Gamma(\mbox{\sc x,y})$.  
{ \begin{figure} [[htbp]
 \begin{center}
\includegraphics[width=1\textwidth,angle=0]{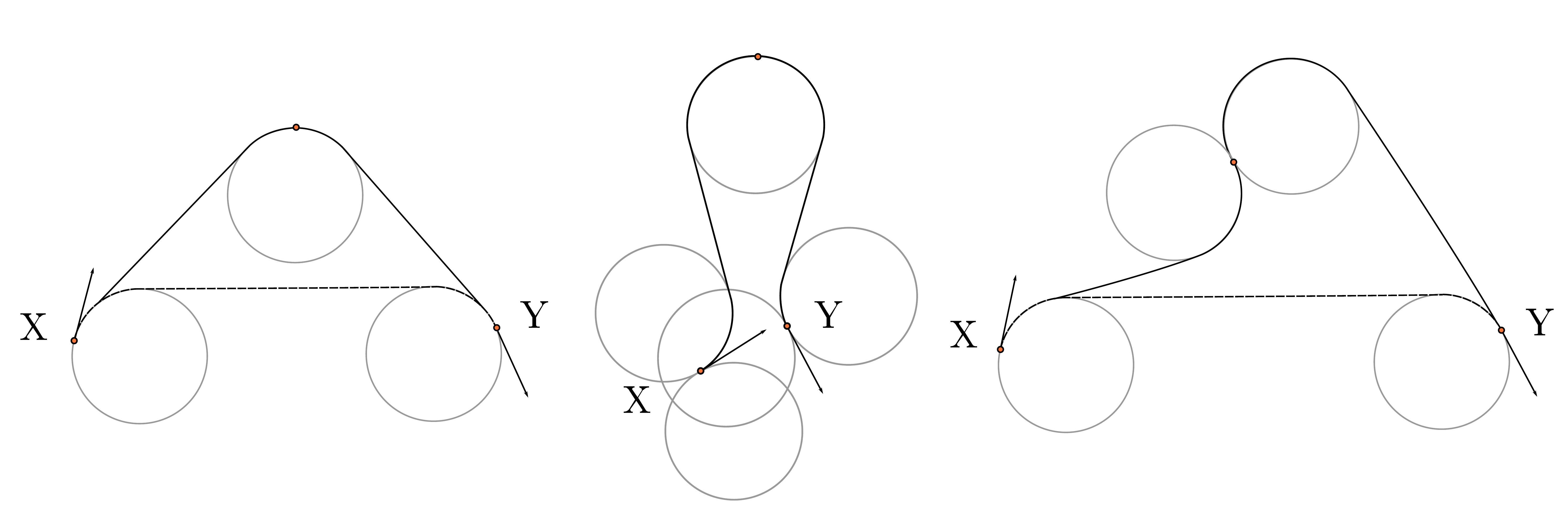}
\end{center}
\caption{Examples of components of type ${\mathscr C}_1$ and ${\mathscr C}_2$. The dashed trace at the left and right illustration represent {\sc csc} paths. The middle illustration correspond to a non-admissible component of type ${\mathscr C}_1$.}
 \label{figrep1}
\end{figure}}

\begin{theorem} \label{embdub} (Dubins \cite{dubins 1}) (Theorem 3.9. in  \cite{papera}) Choose $\mbox{\sc x,y} \in T{\mathbb R}^2$. The minimal length bounded curvature path in $\Gamma(\mbox{\sc x,y})$ is either a {\sc ccc} path having its middle component of length greater than $\pi $ or a {\sc csc} path where some of the circular arcs or line segments can have zero length.
\end{theorem}

The paths in Theorem \ref{embdub} are called {\it Dubins paths} in honour to Lester Dubins who showed for the first time Theorem \ref{embdub} in 1957 in \cite{dubins 1}.

 \section{Reducing paths with loops} \label{windingnumberofpaths}

It is not hard to see that the paths $\gamma_1$ and $\gamma_2$ are not bounded-homotopic (see Figure \ref{figccproxcond1}). In order to describe why the previous idea it is true we introduce the following definition. Consider the exponential map $\exp: {\mathbb R} \rightarrow {\mathbb S}^1$.

\begin{definition} The {\it turning map} $\tau$ is defined in the following diagram,
\[ \xymatrix{ I  \ar[d]_{\tau}   \ar@{>}[dr]^{w} &  \\
                     {\mathbb R} \ar[r]_{\exp}  & {\mathbb S}^1} \]
The map $w:I \rightarrow { \mathbb S}^1$ is called the {\it direction map} and gives the derivative $\gamma'(t)$ of the path $\gamma$ at $t\in I$. The turning map $\tau:I\rightarrow {\mathbb R}$ gives the turning angle $\gamma'(t)$ makes with respect to the $x$-axis.
\end{definition}

Observe that the turning maps of two elements in $\Gamma(\mbox{\sc x,y})$ must differ by an integer multiple called the {\it turning number} denoted by $\tau(\gamma)$. In addition, note there is a one to one correspondence between homotopy classes and these multiples.

  \begin{definition} Given $\mbox{\sc x,y}\in T{\mathbb R}^2$. The space of bounded curvature paths satisfying the given endpoint condition having turning number $n$ is denoted by
$${\Gamma}(n)=\{\gamma \in {\Gamma}(\mbox{\sc x,y}) \,|\,\,\, \tau(\gamma)=n,\,n\in {\mathbb Z}\}.$$
\end{definition}

By virtue of the Graustein-Whitney theorem in \cite{whitney} we have the following.

\begin{corollary} \label{wcpm} Given $\mbox{\sc x,y}\in T{\mathbb R}^2$ then ${\Gamma}(m)\cap {\Gamma}(n)=\emptyset$ for $m\neq n$.
\end{corollary}

Another theorem of Whitney in \cite{whitney} establishes that any curve can be perturbed to have a finite number of transversal self intersections. In \cite{paperd} we developed a continuity argument guaranteeing that the bound on the curvature is never violated when homotoping (projecting) a small fragment onto the replacement path (see Figure \ref{figfunlem}). The number of transversal self intersections in $\gamma$ can be chosen to be minimal and may be denoted by $\chi$. In order to describe a meaningful reduction process for bounded curvature paths in homotopy classes we need to develop a method for lowering the complexity of {\it components with loops} while the length of such components is never increasing through the process. The Corollary 2.4 in \cite{papere} is of technical nature; it basically proves that a bounded curvature path $\gamma$ making u-turn and $d(\gamma(0),\gamma(s))<1$ must have length at least $2$. The end here is to prove that {\it the shortest closed bounded curvature path is the boundary of a unit disk}.

\begin{proposition}\label{noclor} There are no closed bounded curvature paths lying in the interior of a unit radius disk.
\end{proposition}
\begin{proof} Let  $\gamma$ be a closed bounded curvature path lying in the interior of a unit radius disk. By considering $\gamma(0)$ and $\gamma(s)$ as in Corollary 2.4 in \cite{papere} we conclude that $\gamma$ has a pair of points which are distant at least $2$. So, there exists $t_1,t_2\in I$ such that $d(\gamma(t_1),\gamma(t_2))\geq 2$, implying that the path cannot be contained in the interior of a unit disk.
\end{proof}

The next result gives a trivial lower bound for the length of a bounded curvature path (see Figure \ref{figfunlem}). Consider a curve $\gamma (t)=(r(t)\cos \theta(t), r(t)  \sin \theta(t))$ in polar coordinates.

\begin{lemma}\label{rad} (Lemma 2.5 in \cite{papera}). For any curve $\gamma :[0,s] \rightarrow {\mathbb R}^2$ with $\gamma(0)= (1,0)$, $r(t)\geq 1$, and $\theta(s)=\eta$, one has ${\mathcal L}(\gamma)\geq \eta$.
\end{lemma}

\begin{theorem} \label{loopbound} The shortest closed bounded curvature path is the boundary of a unit disk. In particular, closed bounded curvature paths have length at least $2\pi$.
\end{theorem}

\begin{proof}  Let $\gamma$ be a closed bounded curvature path. By Proposition \ref{noclor} there are no closed bounded curvature paths lying in the interior of a unit disk. By the Pestov-Ionin Lemma see \cite{pestov}, {\it a closed bounded curvature path contains a unit disk in its interior component}. Denote by $C$ the boundary of such a disk and consider a coordinate system with origin the centre $o$ of $C$. By applying Lemma \ref{rad} to $\gamma$ with respect to $o$ we conclude that the length of every closed bounded curvature path must be at least $2 \pi$ concluding the proof.\end{proof}

 Let $\gamma$ be a bounded curvature path with self intersections. Observe that there are four possible ways for $\gamma$ to intersects itself transversally for the first time (see Figure \ref{figfirstkink}).

{ \begin{figure} [[htbp]
 \begin{center}
\includegraphics[width=1\textwidth,angle=0]{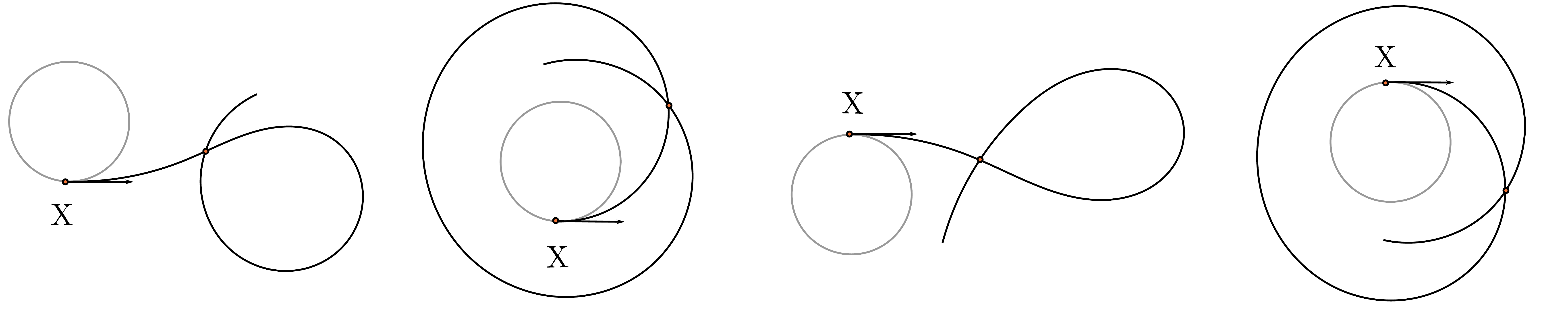}
\end{center}
\caption{The four types of first self intersection.}
 \label{figfirstkink}
\end{figure}}

\begin{definition} Suppose that $\gamma$ intersects itself for first time at $\gamma(t_1)=\gamma(t_2)$ with $t_1<t_2$. The restriction of $\gamma$ to the interval $[t_1,t_2]$ is a {\it loop}. Given $\delta>0$, the restriction of $\gamma$ to the interval $[t_1-\delta,t_2+\delta]$ is a {\it kink}.
\end{definition}

\begin{corollary}\label{looplength} The length of a loop is at least $2 \pi $.
\end{corollary}
 \begin{proof}  Let $\gamma$ be a loop. By a slightly more general version of Pestov-Ionin Lemma see \cite {hee}, {\it a loop contains a unit disk in its interior component}. Denote by $C$ the boundary of such a disk and consider a coordinate system with origin the centre $o$ of $C$. By applying Lemma \ref{rad} to $\gamma$ with respect to $o$ we conclude that the length of the loop must be at least $2 \pi$.
\end{proof}

\begin{definition} \label{c3comp}A component of type ${\mathscr C}_3$ corresponds to a {\sc cscsc} path with a loop (see Figure \ref{figrep3scsnomin}).  A component of type ${\mathscr C}_3$ is called admissible if it satisfies proximity condition {\rm A, B} or {\rm D}. A component of type ${\mathscr C}_3$ is called non-admissible if it satisfies proximity condition {\rm C}. A component as the one at the centre in Figure \ref{figrep3scsnomin} is called degenerate.
\end{definition}
\vspace{-0.3cm}

{\begin{figure} [[htbp]
\begin{center}
\includegraphics[width=1\textwidth,angle=0]{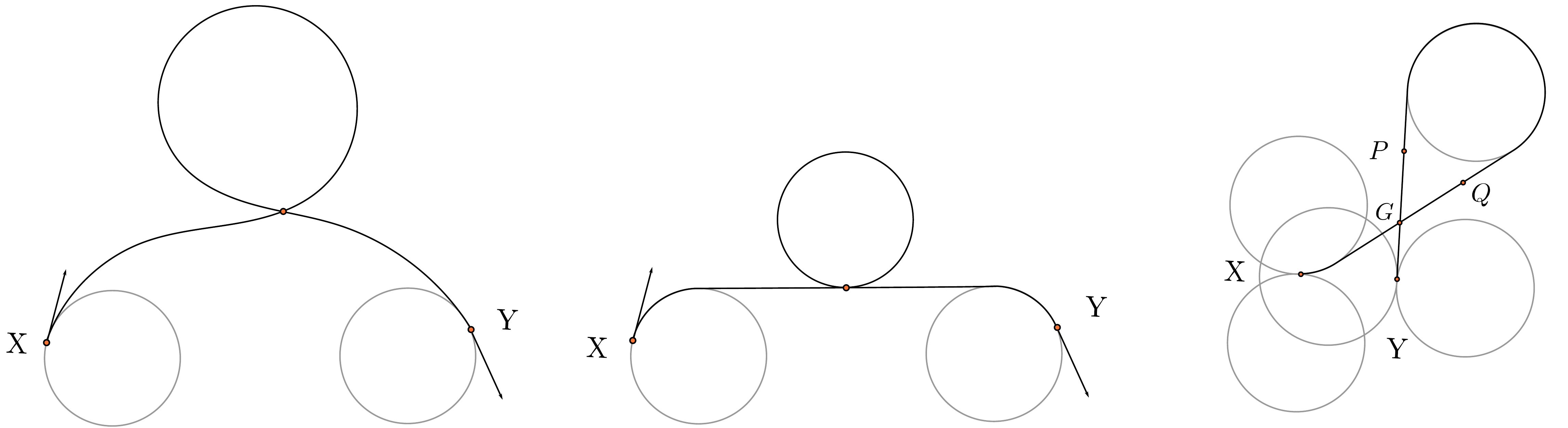}
\end{center}
\caption{Left: A path with a loop. Centre: An admissible degenerate component of type ${\mathscr C}_3$. Right: A  non-admissible component of type ${\mathscr C}_3$.}
\label{figrep3scsnomin}
\end{figure}}

\begin{proposition}\label{lengthred}{\it (Proposition 3.3 in \cite{papera})} Given $\mbox{\sc x,y} \in T{\mathbb R}^2$. A $\it cs$ path in $\Gamma(\mbox{\sc x,y})$ containing an admissible component as a sub path is bounded-homotopic to another $\it cs$ path in $\Gamma(\mbox{\sc x,y})$ with less complexity being the length of the latter at most the length of the former.
\end{proposition}

\begin{proposition} \label{kink} A kink satisfying proximity condition {\sc A}, {\sc B} or {\sc D} is bounded-homotpic to a component of type ${\mathscr C}_3$ of length at most the length of the kink.
\end{proposition}

\begin{proof} Let $\gamma$ be a normalisation of a kink. Let $\gamma(t_1)=\gamma(t_2)$ for $t_1<t_2$ be the self intersection of $\gamma$ and suppose that for some $\delta>0$ the kink satisfies proximity condition {\sc A}, {\sc B} or {\sc D}. By applying a similar argument as in Proposition \ref{lengthred} to the interval $[t_1-\delta,t_2+\delta]$, we have that
$${\mathcal L(\gamma, t_1-\delta, t_1)}+{\mathcal L(\gamma, t_2, t_2+\delta)}>{\mathcal L(\beta)}$$
where $\beta$ is the {\sc csc} replacement path constructed between $\gamma(t_1-\delta)$ and $\gamma(t_2+\delta)$ as in Proposition \ref{construct}. On the other hand, by Corollary \ref{looplength} we have that
$${\mathcal L}(\gamma, t_1, t_2)\geq 2 \pi. $$
 Therefore we have that
 $${\mathcal L}(\gamma, t_1-\delta, t_2+\delta)\geq{\mathcal L(\beta)}+2 \pi. $$

Here ${\mathcal L(\beta)}+2 \pi $ corresponds to the length of a ${\mathscr C}_3$ component. Since $\gamma$ and the component of type ${\mathscr C}_3$ are paths of piecewise constant curvature, it is easy to see that these paths are bounded-homotopic since both paths have the same turning number and satisfy proximity condition {\sc A}, {\sc B} or {\sc D}.
 \end{proof}

\section {Length Minimising Bounded Curvature Paths in $\Gamma(n)$}

\begin{proposition} \label{c3nomin}A non-degenerate component of type ${\mathscr C}_3$ is not a path of minimal length.
\end{proposition}

\begin{proof} Consider a non-degenerate component of type ${\mathscr C}_3$. Let $G$ be the self intersection in the component and consider the points $P$ and $Q$ as shown at the right in Figure \ref{figrep3scsnomin}. By applying Theorem \ref{scsccsnomin} to the non-degenerate component of type ${\mathscr C}_3$ with respect to the points $P$ and $Q$ we constructed a shorter non-degenerate component of type ${\mathscr C}_3$ concluding the proof.
\end{proof}
\begin{definition} We denote by $\mbox{\sc c}^\chi $ a unit circle traversed $2 \chi  \pi $ times.
\end{definition}

Some $cs$ paths can be presented in several ways. For example, a $\mbox{\sc r}^{\chi } \mbox{\sc s}\mbox{\sc r}$ path can also be presented as $\mbox{\sc r} \mbox{\sc s}\mbox{\sc r}^{\chi }\mbox{\sc s}\mbox{\sc r}$ having both paths the same length (see Figure \ref{figsingdubpath1}). In addition, a $cs$ path is called {\it symmetric} if the arcs lying in the adjacent circles have the same orientation, otherwise the path is called {\it skew}. In the next result we present the paths having minimal complexity. In addition, the circle traversed $2 \chi  \pi $ times is placed at the beginning (when possible).

\begin{theorem} \label{singudub} Given $\mbox{\sc x,y} \in T{\mathbb R}^2$ and $n\in{\mathbb Z}$. The length minimiser in $\Gamma(n)$ must be of one of the following types:

\begin{itemize}
\item {\sc csc} or {\sc ccc}.
\item Symmetric $\mbox{\sc c}^\chi \mbox{\sc s}\mbox{\sc c}$ or $\mbox{\sc c}^{\chi}\mbox{\sc c} \mbox{\sc s}\mbox{\sc c}$.
\item Skew $\mbox{\sc c}^{\chi}\mbox{\sc s}\mbox{\sc c}$ or  $\mbox{\sc c}\mbox{\sc s}\mbox{\sc c}^{\chi}$.
\item $\mbox{\sc c}^\chi \mbox{\sc c}\mbox{\sc c}$ or $\mbox{\sc c} \mbox{\sc c}^{\chi}\mbox{\sc c}$.
 \end{itemize}
Here $\chi$ is the minimal number of crossings for paths in $\Gamma(n)$. In addition, some of the circular arcs or line segments may have zero length. In particular, we have Theorem \ref{embdub} in the homotopy class containing the length minimiser in $\Gamma(\mbox{\sc x},\mbox{\sc y})$ (See Figure \ref{figsingmov}).
\end{theorem}

\begin{proof} Consider a path $\gamma \in \Gamma(n)$ and a fragmentation for $\gamma$. By applying Theorem \ref{homotarg} we obtain a normalisation of $\gamma$ of length at most the length of $\gamma$. Consider the first self intersection of the normalisation of $\gamma$. By recursively applying Proposition \ref{kink} we replace the kinks by degenerates components of type ${\mathscr C}_3$ (see Figure \ref{figrep3scsnomin} centre). Since translations are isometries, we slide the middle circle in the component of type ${\mathscr C}_3$ along the $cs$ path and place it to be tangent to $X$ (or $Y$) see Figure \ref{figsingmovccc}. By continuing with this procedure until all the loops are tangent to $X$ (or $Y$) and after cancelling out oppositely oriented loops we end up with a circle traversed $2 \chi  \pi $ times concatenated with a $cs$ path without loops. Then by applying Theorem \ref{embdub} to the $cs$ path, we obtain the minimal length path as one of the six {\sc csc} or {\sc ccc} paths (concatenated with a $\mbox{\sc c}^{\chi}$). Figure \ref{figsingmov} illustrates a {\sc csc} path concatenated with a degenerated $\mbox{\sc c}^{\chi}$. Figure \ref{figsingmovccc} illustrates a {\sc ccc} path concatenated with a non-degenerate $\mbox{\sc c}^{\chi}$. There is one case left. After reducing the fragmentation we may end up with a non-admissible component of type ${\mathscr C}_3$ which by Proposition \ref{c3nomin} it is not a path of minimal length. Depending on the  endpoint condition we may homotope such a component to a {\sc csc} or {\sc ccc} path (see Figure \ref{figrepclosesingtwo}) or to a path of higher complexity. It is not hard to see that the higher complexity path is not a path of minimal length. 
We then homotope the path to a {\sc csc} concatenated with a $\mbox{\sc c}^{\chi}$ component (possibly degenerated) or to a $\mbox{\sc c}^\chi \mbox{\sc c}\mbox{\sc c}$ or $\mbox{\sc c} \mbox{\sc c}^{\chi}\mbox{\sc c}$ (possibly degenerated) see Figure \ref{figrepclosesingtwo}.
 \end{proof}
{ \begin{figure} [[htbp]
 \begin{center}
\includegraphics[width=1.0\textwidth,angle=0]{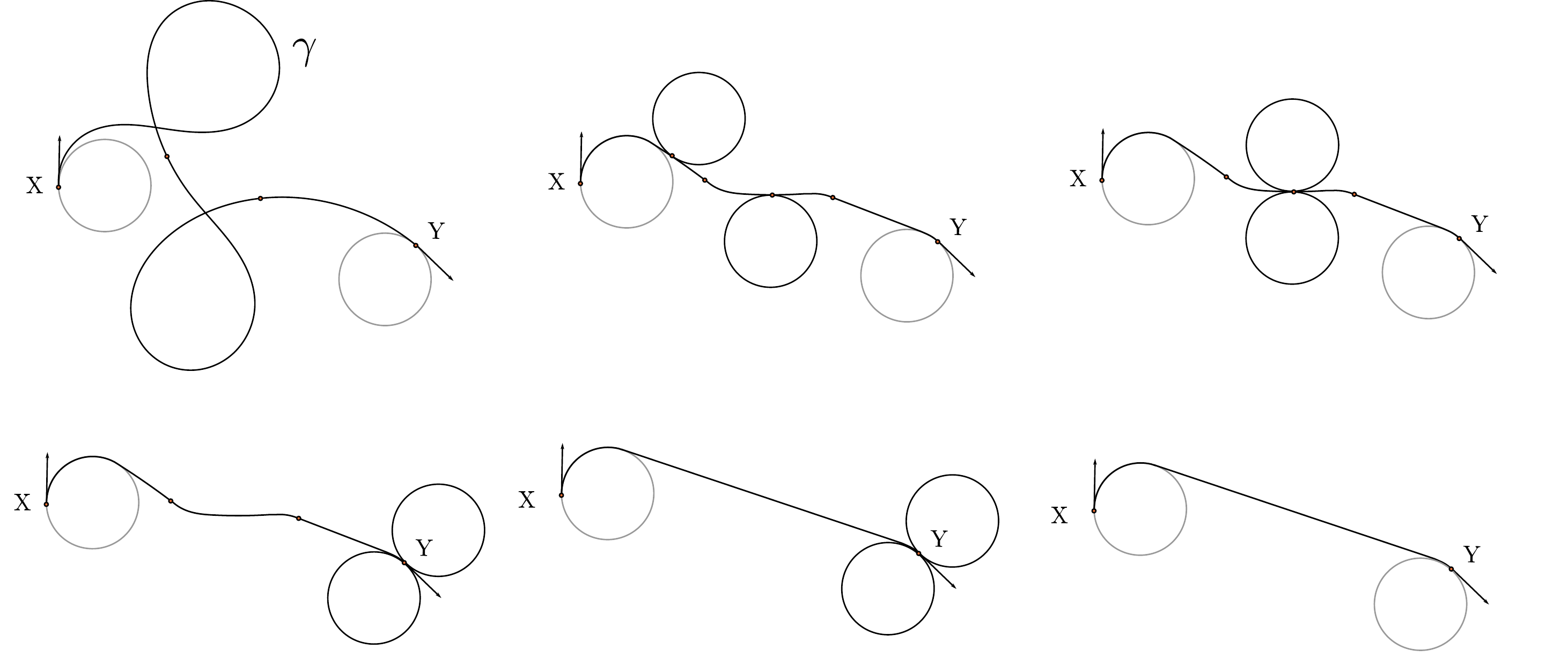}
\end{center}
\caption{By applying Theorem \ref{singudub} to $\gamma$ we obtain the minimal length element in $\Gamma(1)$. Such a {\sc csc} path is indeed the length minimiser in $\Gamma(\mbox{\sc x,y})$. Other length minimisers can be seen in Figures \ref{figsingdubpath1}, \ref{figsingdubpath2} or \ref{figsingdubpath3}.}
\label{figsingmov}
\end{figure}}

{ \begin{figure} [[htbp]
 \begin{center}
\includegraphics[width=1.5\textwidth,angle=90]{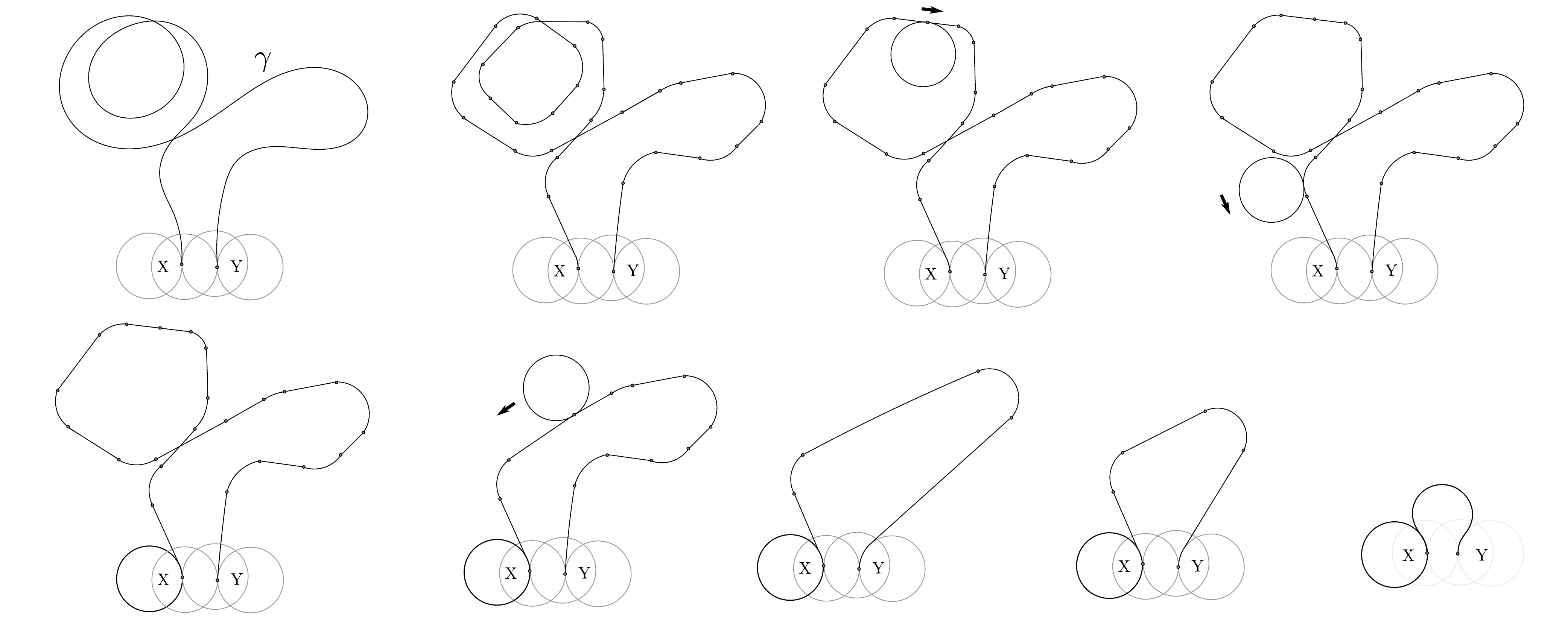}
\end{center}
\caption{An illustration of Theorem \ref{singudub} applied to a path in $\Gamma(-1)$ to obtain the length minimiser in its homotopy class. The global minimum of length in $\Gamma(\mbox{\sc x},\mbox{\sc y})$ lies in $\Gamma(1)$.}
\label{figsingmovccc}
\end{figure}}

{ \begin{figure} [[htbp]
 \begin{center}
\includegraphics[width=1\textwidth,angle=0]{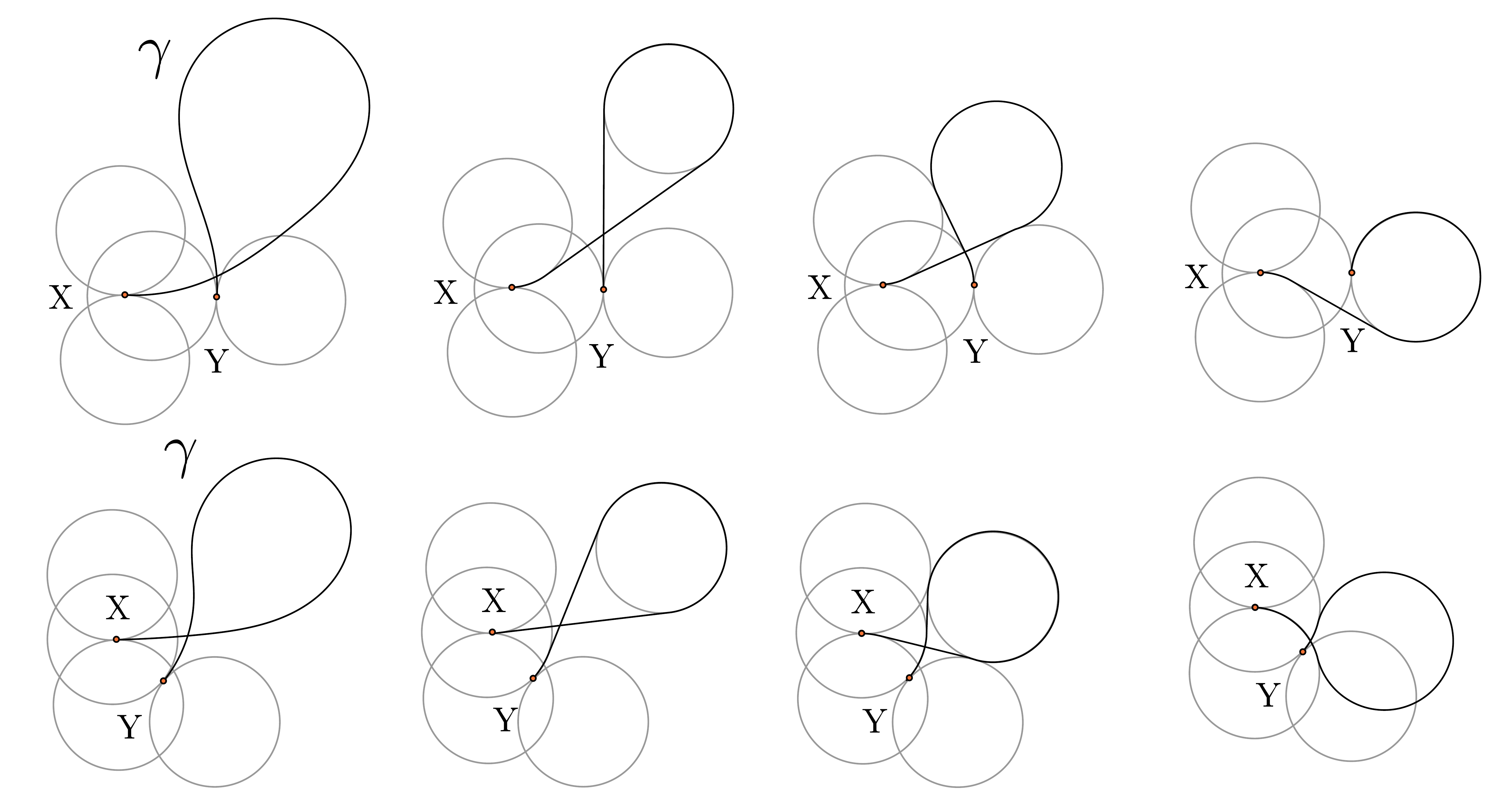}
\end{center}
\caption{Top: The length minimiser in the homotopy class of $\gamma$ is a $\mbox{\sc c}\mbox{\sc s}\mbox{\sc c}$. Bottom: The length minimiser in the homotopy class of $\gamma$ is a $\mbox{\sc c} \mbox{\sc c}\mbox{\sc c}$ path with ${\chi}=1$; such a path is also the global minimum of length.}
\label{figrepclosesingtwo}
\end{figure}}

{ \begin{figure} [[htbp]
 \begin{center}
\includegraphics[width=1\textwidth,angle=0]{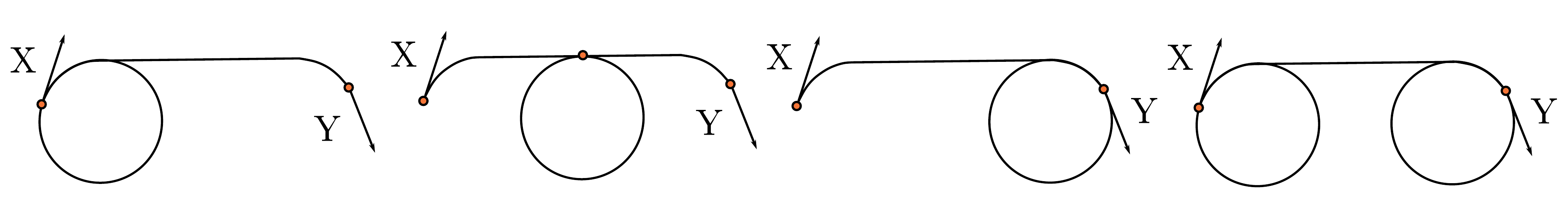}
\end{center}
\caption{Examples of minimal length paths bounded-homotopic to a $\mbox{\sc r}^\chi \mbox{\sc s}\mbox{\sc r}$ path.}
\label{figsingdubpath1}
\end{figure}}

{ \begin{figure} [[htbp]
 \begin{center}
\includegraphics[width=1\textwidth,angle=0]{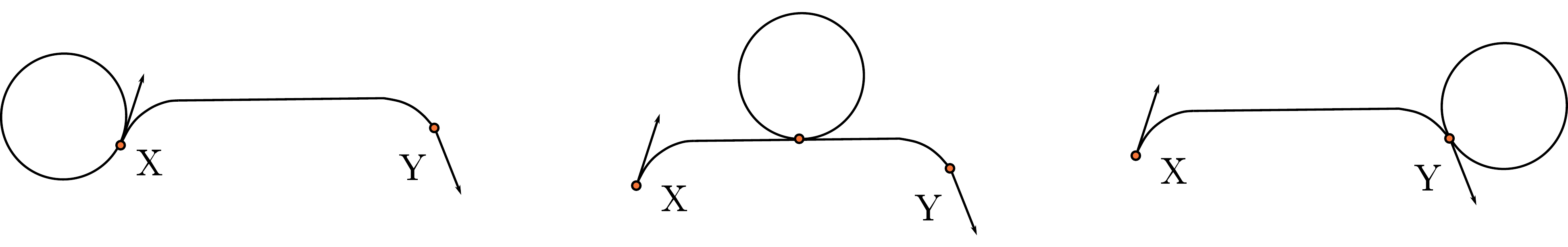}
\end{center}
\caption{Examples of minimal length paths bounded-homotopic to a $\mbox{\sc l}^{\chi}\mbox{\sc r}\mbox{\sc s}\mbox{\sc r}$ path.}
\label{figsingdubpath2}
\end{figure}}

{ \begin{figure} [[htbp]
 \begin{center}
\includegraphics[width=1\textwidth,angle=0]{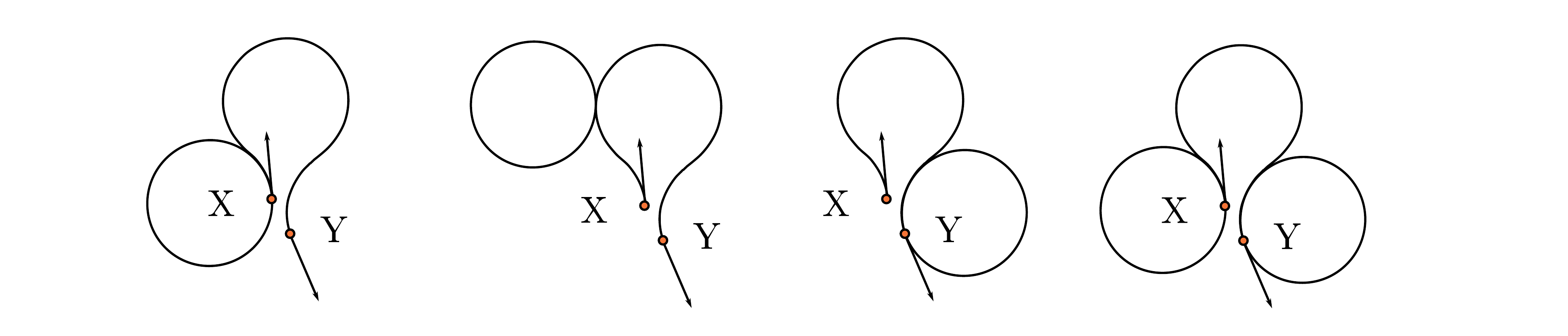}
\end{center}
\caption{Examples of minimal length paths bounded-homotopic to a $\mbox{\sc l}^\chi \mbox{\sc r}\mbox{\sc l}$ path.}
\label{figsingdubpath3}
\end{figure}}

\bibliographystyle{amsplain}
   
\end{document}